\documentclass[11pt]{amsart}
\usepackage{amsmath,amsthm,latexsym,amssymb}
\usepackage{color}
\usepackage{graphicx}
\newcommand{\EEE}{\color{black}}

\numberwithin{equation}{section}

\newtheorem{theorem}{Theorem}
\newtheorem{lemma}{Lemma}

\newtheorem{proposition}{Proposition}
\newtheorem{remark}{Remark}

\numberwithin{theorem}{section}

\numberwithin{corollary}{section}
\numberwithin{lemma}{section}
\numberwithin{definition}{section}
\numberwithin{proposition}{section}
\numberwithin{remark}{section}

\newcommand{\rn}{\mathbb R^n}
\newcommand{\R}{\mathbb R}

\newcommand{\medint}{-\kern  -,375cm\int}

\newcommand{\h}{d\mathcal{H}^{n-1}}

\begin{document}

\title[\textsf{Robin negative parameter}]{A Sharp estimate for the first Robin-Laplacian eigenvalue with negative boundary parameter}
\author[ D. Bucur - V. Ferone -C. Nitsch  -  C. Trombetti]{  Dorin Bucur $^{*}$-Vincenzo ferone$^{**}$ - Carlo Nitsch$^{**}$  -  Cristina Trombetti$^{**}$}
\thanks{$^{*}$Univ. Grenoble Alpes, Univ. Savoie Mont Blanc, CNRS, LAMA 
73000 Chamb\'ery, France,
e-mail: dorin.bucur@univ-savoie.fr \\
$^{**}$ Dipartimento di Matematica e Applicazioni ``R. Caccioppoli'', Universit\`{a}
degli Studi di Napoli ``Federico II'', Complesso Monte S. Angelo, via Cintia
- 80126 Napoli, Italy; email: ferone@unina.it; c.nitsch@unina.it;
cristina@unina.it
}


\begin{abstract}
In this paper we prove that the ball maximizes the first eigenvalue of the Robin Laplacian operator with negative boundary parameter, among all convex sets of $\R^n$ with prescribed perimeter. The key of the proof is a dearrangement procedure of the first eigenfunction of the ball on the level sets of the distance function to the boundary of the convex set, which controls the boundary and the volume energies of the Rayleigh quotient.
\end{abstract}
\maketitle
\section{Introduction}

Let $\Omega$ be a bounded convex subset of $\rn$. 
In this paper we deal with the following  eigenvalue problem for the  Laplacian operator
\begin{equation}\label{P}
\left\{
\begin{array}{ll}
-\Delta u= \lambda(\alpha,\Omega)  u& \mbox{in $\Omega$}\\\\
\dfrac{\partial u}{\partial \nu} +\alpha \,u =0 & \mbox{on $\partial\Omega$.}
\end{array}
\right.
\end{equation}
where $\alpha    <   0$ is a negative parameter,  $\nu$  denotes the outer unit normal to $\partial \Omega$. 

 The fundamental eigenvalue of the Robin-Laplacian on $\Omega$ is defined by 
\begin{equation}
\label{ray}
\lambda(\alpha,\Omega)  = \mathop{\min_{v\in H^1(\Omega)}}_{ v\ne 0} \dfrac {\displaystyle \int_\Omega |\nabla v|^2 \, dx + \alpha \displaystyle\int_{\partial \Omega} v^2 \,  \h }{\displaystyle\int_\Omega  v^2 \, dx}= \mathop{\min_{v\in H^1(\Omega)}}_{ v\ne 0} F(v,\Omega).
\end{equation}

A minimizer $u$ in \eqref{ray} satisfies the equation \eqref{P}  in the weak form
$$\int_\Omega \nabla u \nabla v \,dx +\alpha \int_{\partial \Omega} uv  \h = \lambda(\alpha,\Omega)\int_\Omega uv \,dx, \quad \forall v \in H^1(\Omega). $$

The main objective of the paper is to prove that if $\Omega$ is a convex set and $B$ is a ball, both of them having boundaries of the same $(n-1)$ Hausdorff measure, then
$$\lambda(\alpha,\Omega)\le \lambda(\alpha,B),$$
with equality only if $\Omega$ is a ball. 

This question is related to the long lasting conjecture of Bareket from 1977 (see \cite{B2}) claiming that the ball maximizes $\lambda(\alpha, \Omega)$ among all Lipschitz sets with prescribed volume. Freitas and Krejcirik disproved the conjecture in 2015, giving a counter-example based on the asymptotic behaviour of the eigenvalues on a disc and an annulus of the same area, when $\alpha \rightarrow -\infty$. In the same time, for sets of area equal to $1$, they proved that the conjecture is true, provided $\alpha$ is close to $0$. In any dimension of the space, only the local maximality of the ball is known (see \cite{FNT}). 

In 2017, Antunes, Freitas and Krejcirik studied (see \cite{AFK}) the problem of maximising the first eigenvalue under a perimeter constraint, proving that the disc is the solution among all $C^2$ domains of $\R^2$. Their proof is based on a comparison argument obtained by the method of parallel coordinates, originally introduced by Payne and Weinberger \cite{PW}. The question remained open in arbitrary dimension.

The purpose of this paper is to prove that the inequality holds true in $\R^n$, provided that we restrict ourselfs to the class of convex sets or, even more, to Lipschitz sets which can be written as $\Omega \setminus K$, where $\Omega$ is open and convex, and $K$ is a closed set. Our proof is inspired by the  parallel coordinates method. The idea is to build a suitable test function on a convex set by dearranging the first eigenfunction of the ball onto the level sets of the distance function to the boundary of the convex set. The dearrangement is perfomed in such a way, that the $L^2$-norms of the gradient and the function itself are non-increasing and the $L^2$-norm of the trace at the boundary remains constant. Then, the conclusion comes out from the Rayleigh quotient formulation of the first eigenvalue.  

Several open problems are discussed in the last section. In particular, the inequality we prove in this paper for convex sets is conjectured to hold for {\it every } set. We prove an existence of an optimal set for the shape optimization problem maximzing $\lambda$ in the class of all measurable sets with perimeter not larger than a given constant. 
The maximality of the ball for the first eigenvalue in the class of convex sets, but under a volume constraint, is another  a challenging problem which, for the moment, remains also open.

\section{Notation and preliminaries}
Throughout the paper we
 denote by $|D|$ and $P(D)$ the $n$-dimensional Lebesgue measure of D and the $(n-1)$-dimensional Hausdorff measure in $\rn$ of its boundary. Moreover,  $B_1$ stands for the open unit ball in $\rn$ and $\omega_n=|B_1|$.

\begin{proposition}
\label{simplicity}
The first eigenvalue $\lambda(\alpha,\Omega)$ of \eqref{P} is non positive and it is simple, that is all the associated eigenfunctions are scalar multiples of each other, and the corresponding eigenfunction maybe taken to be positive.
\end{proposition}

\begin{proof}
Observe that, choosing the constant as test function in \eqref{ray}, it results
$$
\lambda(\alpha,\Omega)   <\alpha \frac{P(\Omega)}{|\Omega|} <0. 
$$
Since $\forall v \in H^1(\Omega)$ it results $F(v,\Omega) \ge F(|v|,\Omega)$, with equality if and only if $v$ has constant sign in $\Omega$, then any eigenfunction associated to $\lambda(\alpha,\Omega) $ must have constant sign.
The simplicity then follows by standard arguments.
\end{proof}

Let $K$ be a nonempty, compact, convex set (i.e. a convex body) and let $\rho>0$. Then the Steiner formula     for the perimeter  it reads as
\begin{eqnarray}
\label{per}
P(K+\rho B_1)&=&n\sum_{i=0}^{n-1}\binom{n-1}{i}W_{i+1}(K)\rho^i
\\
&=& P(K)+n(n-1)W_2(K)\rho + ... +nW_n(K)\rho^n. \notag
\end{eqnarray}
It immediately follows that
\begin{equation}\label{derivata}
\lim_{\rho \to 0^+}\frac{P(K+\rho B_1)-P(K)}{\rho}=n(n-1)W_2(K).
\end{equation}
When $K$ is of class $C^2_+$, \eqref{per} and \eqref{derivata} imply
\begin{equation*}
\lim_{\rho \to 0^+}\frac{P(K+\rho B_1)-P(K)}{\rho}=(n-1)\int_{\partial K} H \h.
\end{equation*}
where, denoted by  $\kappa_1,...,\kappa_{n-1}$ are the principal curvatures of of $\partial K$,
$$
H={(n-1)}^{-1} \sum_{1 \le i \le n-1}\kappa_{i}
$$
is the mean curvature of $\partial K$.
Finally the following Aleksandrov-Fenchel inequality holds true

\begin{equation}\label{AF2}
W_2(K) \ge n^{-\frac{n-2}{n-1}}{\omega_n}^{\frac{1}{n-1}}P(K)^{\frac{n-2}{n-1}}.
\end{equation}
In what follows, the notions introduced above applied to an open set have to be 
understood applied to its closure. 

\medskip
\section{Main Result}
From now on  we denote   
 $$\Omega_t = \{  x \in \Omega : d(x)>t\},\qquad t \in [0,r_\Omega],$$
  where $d(x)$ is the distance of a point $x \in \Omega$ from the boundary of $\Omega$ and $r_\Omega$ is the inradius of $\Omega$. The Brunn-Minkowski theorem (see, for example, \cite{BZ}, \cite{S}) ensures that the map $ (P(\Omega_t))^{\frac{1}{n-1}}$ is concave in $[0,r_\Omega]$. This implies that $ (P(\Omega_t))^{\frac{1}{n-1}}$ is an absolutely continuous function in $]0,r_\Omega[$, there also exists its   right derivative  at $0$ and, since
   $(P(\Omega_t))^{\frac{1}{n-1}}$ is strictly monotone decreasing in $[0,r_\Omega[$, such a derivative is negative.
  
\begin{lemma}
Let $\Omega$ be a bounded, convex, open set in $\rn$. Then, for almost every $t \in ]0,r_\Omega[$,
$$-\frac{d}{dt} P(\Omega_t) \ge n(n-1)W_2(\Omega_t),$$
equality holding if $\Omega$ is a ball.
\end{lemma}

\begin{proof}
It is easy to check that 
$$
\Omega_t +\rho B_1 \subset \Omega_{t-\rho},\qquad 0<\rho<t
$$
and, when $\Omega$ is a ball, these two sets coincide. Since the perimeter is monotone  with respect to the inclusion of convex sets, we get, for almost every $t \in ]0,r_\Omega[$,
$$
-\frac{d}{dt} P(\Omega_t) = \lim_{\rho\to 0^+} \frac{P(\Omega_{t-\rho}) - P(\Omega_t)}{\rho}\ge
 \lim_{\rho\to 0^+} \frac{P(\Omega_t+\rho B_1) - P(\Omega_t)}{\rho} =n (n-1)
 W_2(\Omega_t).
$$
\end{proof}

By simply applying the chain rule formula and recalling that $|Dd(x)|=1$ almost everywhere, it is immediate to prove the following

\begin{lemma}\label{f(d)}
Let $u(x) = f(d(x))$, where $f:[0,+\infty[ \to [0,+\infty[$ is a non increasing, $C^1$ function. Setting
$$
E_t = \{ x \in \Omega: u(x)<t\} = \Omega_{f^{-1}(t)};
$$
then 
$$
\frac{d}{dt} P(E_t) \ge (n-1)\frac{W_2(E_t)}{|Du|_{u=t}}.
$$
\end{lemma}

\noindent Now we can state our main result.
\begin{theorem}\label{mainthm}
Let $\Omega$ be a bounded, convex, open set in $\rn$ and let $\Omega^\star$ be a ball  with the same perimeter as $\Omega$. Denoted by  $\lambda(\alpha,\Omega)$ and $\lambda(\alpha,\Omega^\star)$  the first eigenvalues of \eqref{P} on $\Omega$ and $\Omega^\star$, then 
\begin{equation}\label{main}
\lambda(\alpha,\Omega) \le \lambda(\alpha,\Omega^\star).
\end{equation}
 Equality holds only if $\Omega$ is a ball. 
\end{theorem}

\begin{remark}\label{palla}
Proposition  \ref{simplicity} ensures that  a positive eigenfunction $v$ associated to the first eigenvalue of \eqref{P} in a ball is radially symmetric and radially increasing.
Moreover $v(x)= v(|x|)=\phi(r)$, where $\phi$ solves
\begin{equation}\label{PR}
\left\{
\begin{array}{ll}
-r^{-(n-1)}[r^{(n-1)} \phi'(r)]'= \lambda(\alpha,B_R) \phi(r) &  r \in[0,R]\\\\
\phi'(0)=0\\\\
\phi'(R) +\alpha \phi(R)=0 .
\end{array}
\right.
\end{equation}
The solution to \eqref{PR} is explicit and it is 
\[
\phi(r) =r^{-\beta}I_{\beta}(kr), \quad \beta= \frac{n-2}{2}
\]
where $I_{\beta}$ is a modified Bessel function (see \cite{AS} Section 9.6) and $k= \sqrt{-\lambda(\alpha,B_R)}$ is the smallest nonnegative root of the equation:
\[
kI_{\beta+1}(kR)+ \alpha I_{\beta}(kR)=0.
\]
\end{remark}

\begin{proof}[Proof of Theorem \ref{mainthm}]

Let $v$ be the eigenfunction associated to $\lambda(\alpha,\Omega^\star)$. We denote by $v_m=v(0) = \min_{\Omega^\star} v$ and by $v_M= \max_{\Omega^\star} v$.
 Remark \ref{palla} ensures that the gradient of $v$ has constant modulus on the level lines of $v$.

  Let us consider the function $g(t)=|Dv|_{v=t}$, $v_m<t \le v_M$. Set  $w(x)=G(d(x))$,  $x \in\Omega$, where $G^{-1}(t)=\displaystyle \int_{t}^{v_M} {\frac{1}{g(s)}}ds$. By construction $w \in H^1(\Omega)$ and it results:
 
 \begin{equation}
 \label{testfunction}
 \begin{array}{ll}
 & w_M= \max_{\Omega} w =v_M= G(0); \cr \\
 &  w_m= \min_{\Omega} w =G(r_\Omega) \geq v_m = G(r_{\Omega^\star}) \cr \\
 &  |Dw|_{w=t} = |Dv|_{v=t} = g(t) \quad w_m \le t \le w_M.
 \end{array}
 \end{equation}

\noindent Let $$
E_t = \{ x \in \Omega : w(x) <t\} =  \{ x \in \Omega : d(x) > G^{-1}(t)\}, \>
B_t = \{ x \in \Omega^\star : v(x) <t\}. 
$$
 By Lemma \ref{f(d)} and the isoperimetric inequality \eqref{AF2} we get for $quad w_m < t \le w_M.$
 \begin{eqnarray*}
 \frac{d}{dt} P(E_t) &\ge&(n-1) \frac{W_2(E_t)}{g(t)} 
 \\
 & \geq &
(n-1)n^{-\frac{n-2}{n-1}} \omega_n^{ \frac{1}{n-1}} \frac{(P(E_t))^{\frac{n-2}{n-1}}}{g(t)} \notag
 \end{eqnarray*}
 while for $v$ it holds
  \begin{equation*}
 \frac{d}{dt} P(B_t) =(n-1)
n^{-\frac{n-2}{n-1}} \omega_n^{ \frac{1}{n-1}} \frac{(P(B_t))^{\frac{n-2}{n-1}}}{g(t)}
 \end{equation*}
 with $P(\Omega)=P(E_{w_M})=P(B_{v_M})=P(\Omega^\star)$. Then, by classical comparison theorems,
\begin{equation}\label{perimetri}
P(E_t) \le P(B_t), \qquad  w_m < t \le w_M.
\end{equation}

On the other hand, by \eqref{perimetri},
$$
\int_{w=t}|Dw|\,\h =g(t) P(E_t)\le g(t) P(B_t)= \int_{v=t}|Dv| \,\h, \qquad w_m < t \le w_M
$$

then, by co-area formula and \eqref{testfunction},
\begin{eqnarray}\label{energie}
\int_\Omega |Dw|^2\,dx&=&\int_{w_m}^{w_M} g(t)P(E_t)\,dt
\\
& \le& \int_{w_m}^{v_M} g(t)P(B_t)\,dt \leq \int_{v_m}^{v_M} g(t)P(B_t)\,dt 
\\
&=& \int_{\Omega^\star}|Dv|^2\,dx. \notag
\end{eqnarray}

Since by construction $w(x) = w_M$ if $ x \in \partial \Omega$, we have
\begin{equation}
\label{boundary}
\int_{\partial \Omega} w^2 \, \h = w_M^2 P(\Omega) = v_M^2 P(\Omega^\star) = \int_{\partial \Omega^\star} v^2 \, \h
\end{equation}
 
 \noindent  Moreover, denoted by $\mu(t)=|\{x \in \Omega:\>  w(x) < t\}| = |E_t|$ and $\nu(t)=|B_t| = |\{x \in \Omega^\star:\> v(x) <t\}|$, by co-area formula we obtain
 \begin{equation*}
 \mu'(t)=\int_{w=t}\frac{1}{|Dw|}\,\h=\frac{P(E_t)}{g(t)} \le \frac{P(B_t)}{g(t)}=\int_{v=t}\frac{1}{|Dv|}\,\h = \nu'(t), w_m < t \le w_M.
 \end{equation*}
 The above inequality holds true (trivially) even if $ 0<t<w_m$. Then, integrating between $s$ and $v_M$ 
 
 \begin{equation}
 |\Omega| - \mu(s) \leq |\Omega^\star| - \nu(s) \quad s \in [0,v_M].
 \end{equation}
 
This implies
\begin{eqnarray}\label{normeL2}
\int_\Omega w^2 \,dx&=& \int_0^{w_M} 2t(|\Omega|-\mu(t))\,dt \leq
\\
&=&  \int_0^{v_M} 2t(|\Omega|-\nu(t))\,dt  = \int_{\Omega^\star} v^2 \,dx\notag
\end{eqnarray}

 By \eqref{ray}, \eqref{energie}, \eqref{normeL2} and \eqref{boundary} we finally have
\begin{eqnarray*}
\lambda(\alpha,\Omega)&\le& \dfrac {\displaystyle \int_\Omega |\nabla w|^2 \, dx + \alpha \displaystyle\int_{\partial \Omega} w^2 \,  \h }{\displaystyle\int_\Omega  w^2 \, dx}\le \notag
\\
& \le & \dfrac {\displaystyle \int_{\Omega^\star} |\nabla v|^2 \, dx + \alpha \displaystyle\int_{\partial \Omega^\star} v^2 \,  \h }{\displaystyle\int_{\Omega^\star}  v^2 \, dx} =
\lambda(\alpha,\Omega^\star)\> <0 \notag
\end{eqnarray*}
 getting the claim.
 
  If equality occurs, then $\Omega$ is a ball as a consequence of the equality case in the isoperimetric inequality. 
 \end{proof}

\section{Further remarks and open questions}
\noindent {\bf More general sets.} Let $\Omega \subset \R^n$ be a bounded, open, convex set and $K \subset \Omega$ be a closed set, smooth enough such that the eigenvalue problem is well defined in $\Omega \setminus K$. For instance, $K$ may be the closure of a Lipschitz set. It can be easily observed that 
$$\lambda(\alpha, \Omega \setminus K) \le \lambda (\alpha ,\Omega).$$ 
Taking  $v$  a non-zero eigenfunction in $\Omega$, one has
$$\lambda(\alpha, \Omega \setminus K) \le\dfrac {\displaystyle \int_{\Omega\setminus K} |\nabla v|^2 \, dx + \alpha \displaystyle\int_{\partial (\Omega\setminus K)} v^2 \,  \h }{\displaystyle\int_{\Omega \setminus K} v^2 \, dx}\le \dfrac {\displaystyle \int_\Omega |\nabla v|^2 \, dx + \alpha \displaystyle\int_{\partial \Omega} v^2 \,  \h }{\displaystyle\int_\Omega  v^2 \, dx}= \lambda (\alpha ,\Omega).
$$
Using Theorem \ref{mainthm} applied to $\Omega$, the fact that the perimeter of $\Omega \setminus K$ is larger than the perimeter of $\Omega$ together to the monotonicity on balls (i.e.    $\lambda (\alpha, B_{r_1}) <\lambda (\alpha, B_{r_1}$), if $r_1 <r_2$,  see for instance \cite[Theorem 5]{AFK}) one gets that the statement of Theorem \ref{mainthm} applies to $\Omega \setminus K$, i.e.
 \begin{equation}\label{main.b}
\lambda(\alpha,\Omega\setminus K) \le \lambda (\alpha ,\Omega)\le \lambda (\alpha ,\Omega^\star)\le  \lambda(\alpha,(\Omega\setminus K)^\star)
\end{equation}

It remains unclear how large is the class of open sets for which Theorem \ref{mainthm} remains true (see Open problem 1, below). For instance, we do not know whether Theorem \ref{mainthm} is valid even in the class of contractible domains. From a shape optimization point of view, there is some similarity in behaviour  between the Robin eigenvalue with negative boundary parameter and the Steklov eigenvalue. In dimension larger than  $2$, performing a small hole in the center of the ball and rescaling the geometry to keep the surface area constant, rises the Steklov eigenvalue. Fraser and Schoen proved that there exists even a contractible domain with higher Steklov eigenvalue  than the ball with the same surface area in any dimension $n\ge 3$ (see \cite{FS}) by a suitable deformation of the punched ball.  

We report and discuss below some open problems listed in  \cite{AFK}, among which we have partially solved the first one.

\noindent {\bf Open problem 1.}
Prove that for $n \ge 3$ and for every $c >0$ the solution of 
 \begin{equation}\label{main.s}
 \sup\{ \lambda(\alpha, \Omega) : \Omega \subset \R^n, \; P(\Omega)= c\},
 \end{equation}
is the ball. 
\EEE

This question is related to the maximal value of the mean curvature, see \cite{PP}. For a smooth set $\Omega$
\[
\lim_{\alpha\to -\infty} \lambda(\alpha, \Omega^\star)-\lambda(\alpha, \Omega) > 0
\]
 if and only if the maximal value of the mean curvature of the boundary of $\Omega$ is larger than that of $\Omega^*$. This is always true. Indeed, for any given set $\Omega$ we can consider the convex hull $C_{\Omega}$ and observe that 
 
\begin{eqnarray*}
n\omega_n &=& \int_{\partial C_{\Omega}} \kappa\,  \h=\int_{\partial C_\Omega\cap \partial\Omega} \kappa\,  \h\\
&\le& \int_{\partial C_\Omega\cap \partial\Omega} \left(\frac{H}{n-1}\right)^{n-1} \, \h\le\frac{P(\Omega)}{(n-1)^{n-1}} \left(\max_{\partial\Omega} H\right)^{n-1} ,
 \end{eqnarray*}
with equality on balls.
Here $\kappa$ is the Gauss curvature of $\partial C_{\Omega}$, (that is the product of the $n-1$ principal curvatures of the boundary of $C_{\Omega}$). We have used the fact that the gaussian curvature of $\partial C_{\Omega}$ vanishes when $\partial C_{\Omega}\setminus\partial\Omega$. Moreover the inequality $\kappa\le \left(\frac{H}{n-1}\right)^{n-1}$ is follows from arithmetic-geometric mean inequality.

The existence of an optimal shape (independent on the knowledge of its geometry) for the shape optimization problem
above
was not known for $n \ge 3$. We shall prove below an existence result in the larger class of measurable sets of finite perimeter. Of course, it would be interesting to prove the regularity of the optimal set, but this question is, in general, very delicate. 

In order to prove an existence result, we have to extend the definition of $\lambda$ to measurable sets with finite perimeter. Let $\Omega \subset \R^n$ be a measurable set with finite measure. The (generalized) perimeter of $\Omega$ is defined by
$$P(\Omega):=\sup\{ \int_\Omega \mbox{div} \varphi \, dx : \varphi \in C^\infty_c (\R^n, \R^n), \|\varphi\|_\infty \le 1\}.$$
If $\Omega$ is Lipschitz, then the perimeter defined above coincides with the $(n-1)$-Hausdorff measure of the topological boundary. If $\Omega$ is just measurable and $P(\Omega) <+\infty$, then the perimeter defined above coincides with the $(n-1)$-Hausdorff measure of the set of points with density $\frac 12$ in $\Omega$. This set is denoted $\partial ^* \Omega$ and is called the reduced boundary (and is rectifiable). The natural extension of the definition of $\lambda$ to measurable sets with finite perimeter is 
\begin{equation}
\label{ray.m}
\lambda(\alpha,\Omega)  = \mathop{\inf_{v\in H^1(\R^n)}}_{ v|_\Omega \ne 0} \dfrac {\displaystyle \int_\Omega |\nabla v|^2 \, dx + \alpha \displaystyle\int_{\partial ^*\Omega} v^2 \,  \h }{\displaystyle\int_\Omega  v^2 \, dx}.
\end{equation}

\begin{proposition}\label{th001}
Let $c>0$ be fixed. The following shape optimization problem has a solution
 \begin{equation}\label{main.s}
 \max\{ \lambda(\alpha, \Omega) : \Omega \subset \R^n, \mbox{ measurable}, \; P( \Omega)\le  c\}.
 \end{equation}

\end{proposition}
Under a volume constraint, this question has been discussed in \cite{BC}. The perimeter constraint does not rise specific difficulties. An intriguing point is that we do not know whether the constraint is saturated at the optimal set $ \Omega_{opt} $, namely $P( \Omega_{opt} )=c$. This fact is quite surprising, and is essentially due to the fact that $\lambda(\alpha, \Omega)$ does not behave in a controllable way to rescaling, namely we do not know if for every $\Omega$ and for every $t \ge 1$ we have
$$\lambda(\alpha, t\Omega) \ge \lambda(\alpha, \Omega).$$
\begin{proof}[Proof of Proposition \ref{th001}] Let $(\Omega_k)_k$ be a maximizing sequence. We have
$$\lambda(\alpha, \Omega_k) \le \alpha \frac{P(\Omega_k)}{|\Omega_k|} \le \alpha C_n \frac{1}{|\Omega_k|^\frac{1}{n} },$$
where $C_n$ is a dimensional constant popping up in the isoperimetric inequality in $\R^n$. Let $A>0$ be equal to $-\lambda(\alpha, B)$, where $B$ is the ball of perimeter equal to $c$. Then, we can assume that the maximizing sequence satisfies $\lambda(\alpha, \Omega_k) > -A$. From this inequality on one hand, and the isoperimetric inequality on the other, we notice that there exists two constant $0<m_1<m_2$ such that $m_1\le |\Omega_k|\le m_2$. 

We recall the following isodiametric control property from \cite[Proposition 14]{BC}, in a simplified form.
\begin{lemma}
 Let $m, A >0$. There exists a constant $D=D(m,\alpha,n,A)$ such that 
if $\Omega\subset\mathbb{R}^d$ is a set of finite perimeter such that $|\Omega|\le m$  and ${\lambda}(\alpha, \Omega) > -A$, then $\Omega$ can be decomposed (up to a set of zero measure) as a union of at most $N$  bounded sets of finite perimeter, pairwise  at  positive distance
$$\Omega=\Omega_1\cup\ldots\cup\Omega_N,$$
where $N<C\frac{m A^n}{|\alpha|^n}+1$ and  $\text{diam}(\Omega_j)\le D$, $C$ being a dimensional constant.
\end{lemma}
A direct consequence of this lemma, is that we can replace the set $\Omega_k$ by one of the sets of its decomposition, say $\tilde \Omega_k$ which satisfies
$$\lambda(\alpha, \Omega_k) =\lambda(\alpha, \tilde \Omega_k), \;\;\; P(\tilde \Omega_k) \le P(\Omega_k), \;\;\; \text{diam}(\tilde \Omega_k)\le D.$$
Up to translations, we can assume tha all sets $\tilde \Omega_k$ lie in a ball $B$ of radius $D$. As a consequence of the weak-$\star$ compactness theorem in $BV(B)$ applied to the sequence $(1_{\tilde \Omega_k})_k$, we can extract a subsequence (still denoted using the same index) such that for some measurable subset denoted $ \Omega_{opt} \subset B$ we have
$$1_{\tilde \Omega_k} \stackrel{L^1(B)}{\longrightarrow} 1_{ \Omega_{opt} }, \;\;\; P( \Omega_{opt} ) \le \liminf_{k \rightarrow +\infty} P(\tilde  \Omega_k).$$
 Following the upper semicontinuity result in \cite[Proposition 16, relation (22)]{BC}, we get 
 $$  \limsup_{k \rightarrow +\infty} \lambda(\alpha,\tilde  \Omega_k)\le \lambda(\alpha,   \Omega_{opt}).$$
 This last inequality proves that $ \Omega_{opt} $ is a solution for \eqref{main.s}. 
\end{proof}

\medskip

\noindent {\bf Open problem 2.}
Prove that for $n \ge 3$ the solution of 
 \begin{equation}\label{main.s1b}
\sup\{ \lambda(\alpha, \Omega) : \Omega \subset \R^n, \mbox {$\Omega$ convex},\; |\Omega|= 1\},
 \end{equation}
is the ball.

This question can also be related to the maximal value of the mean curvature, see \cite{PP}.  In \eqref{main.s1b} it is reasonable to require a convexity constraint. Following \cite{FNT16} there exist in $\R^n$ for $n \ge 3$ smooth domains, diffeomorphic to the ball, of volume equal to $1$ and with maximal mean curvature smaller than the mean curvature of the ball of the same volume. This implies that in this latter class of sets, the ball can not be optimal for any $\alpha <0$. 

\noindent {\bf Open problem 3.}
Prove that for $n =2 $ the solution of 
 \begin{equation}\label{main.s2}
\sup\{ \lambda(\alpha, \Omega) : \Omega \subset \R^2, \mbox {$\Omega$ simply connected},\; |\Omega|= 1\},
 \end{equation}
is the ball.

 The argument based on the mean curvature is not anymore valid to contradict \eqref{main.s2}.
For the shape optimization problem  \eqref{main.s2}  the existence of a solution was proved in \cite{BC}.

\EEE

\bigskip
  \noindent {\bf Acknowledgments.} D.B. was supported by the  "Geometry and Spectral Optimization"  research programme 
LabEx PERSYVAL-Lab GeoSpec (ANR-11-LABX-0025-01) and  Institut Universitaire de France.

\end{document}